\documentclass[letterpaper, 10 pt, conference]{ieeeconf}  

\usepackage{svg}
\usepackage{comment}
\usepackage{lipsum}
\usepackage{graphicx}
\usepackage{amsmath,amsfonts,amssymb,amscd}
\usepackage{amsthm}
\usepackage[colorlinks]{hyperref}
\usepackage{mathtools}
\usepackage{notes2bib}

\newtheorem{cor}{Corollary}
\newtheorem{thm}{Theorem}
\newtheorem{lem}{Lemma}
\newtheorem{proposition}{Proposition}

\usepackage{algorithm}
\usepackage{algpseudocode}

\newcommand{\R}{\mathbb{R}}

\newcommand{\beqn}{\begin{eqnarray*}}
\newcommand{\eeqn}{\end{eqnarray*}}

\newcommand{\gattract}{\mathbb{A}} 

\IEEEoverridecommandlockouts                              

\newcommand{\halmos}{\rule{1ex}{1.4ex}}
\newenvironment{myproof}{\noindent {\em Proof}.\ }{\hspace*{\fill}$\halmos$\medskip}
\newcommand{\epr}{\end{myproof}}
\newcommand{\bpr}{\begin{myproof}}

\title{\LARGE \bf
A remark on omega limit sets for non-expansive dynamics
}

\author{Alon Duvall$^{1}$ and Eduardo D. Sontag$^{2}$
\thanks{This work was partially supported by grants
AFOSR FA9550-21-1-0289 and NSF/DMS-2052455}%
\thanks{$^{1}$Northeastern University
        {\tt\small duvall.a@northeastern.edu}}%
\thanks{$^{2}$Northeastern University
        {\tt\small e.sontag@northeastern.edu, sontag@sontaglab.org}}%
}

\begin{document}

\maketitle
\thispagestyle{empty}

\begin{abstract}
In this paper, we study systems of time-invariant ordinary differential equations whose flows are non-expansive with respect to a norm, meaning that the distance between solutions may not increase. Since non-expansiveness (and contractivity) are norm-dependent notions, the topology of $\omega$-limit sets of solutions may depend on the norm. For example, and at least for systems defined by real-analytic vector fields, the only possible $\omega$-limit sets of systems that are non-expansive with respect to polyhedral norms (such as $\ell^p$ norms with $p =1$ or $p=\infty$) are equilibria. In contrast, for non-expansive systems with respect to Euclidean ($\ell^2$) norm, other limit sets may arise (such as multi-dimensional tori): for example linear harmonic oscillators are non-expansive (and even isometric) flows, yet have periodic orbits as $\omega$-limit sets. This paper shows that the Euclidean linear case is what can be expected in general: for flows that are contractive with respect to any strictly convex norm (such as $\ell^p$ for any $p\not=1,\infty$), and if there is at least one bounded solution, then the $\omega$-limit set of every trajectory is also an omega limit set of a linear time-invariant system.
\end{abstract}

\section{Introduction}

Contraction theory concerns dynamical systems which posses some kind of metric, typically arising from a norm, such that for every two trajectories, their distance is nonincreasing or even decreasing over time. The use of contraction analysis in control theory was pioneered by Slotine and collaborators~\cite{Loh_Slo_00}.
Expositions of contractivity in dynamical systems can be found for example in \cite{7039986}, \cite{Sontag2010}, and \cite{FB-CTDS}, which also show that in general non-Euclidean norms must be considered when analyzing nonlinear dynamics. Most work deals with cases when the distance between trajectories is strictly decreasing, though sometimes the situation arises where all we can say is that this distance is nonincreasing. Our paper is concerned with dynamical conclusions that one can draw when a dynamical system has a merely nonincreasing norm. 

Contraction theory has many connections to control theory and dynamical systems, as well as other fields. It has applications to data-driven control \cite{TSUKAMOTO2021135}, reaction diffusion systems \cite{9640589}, Hopfield neural networks \cite{9961864}, Riemannian manifolds \cite{SIMPSONPORCO201474}, network systems \cite{9403888}, and system safety \cite{jafarpour2023monotonicity}. Establishing contractivity of a system allows one to conclude many desirable stability properties. This makes contraction theory a useful tool in the context of certifying robustness guarantees.

Our main results stated informally are as follows. In the following suppose we are given a system that possesses at least one bounded trajectory. If a system is nonexpansive with respect to some norm, then solutions will converge to a global attractor set on which the system evolves isometrically. The structure of these omega limit sets is dependent on the particular choice of norm for which the system is nonexpansive.  If the norm is strictly convex then the equilibrium set is convex, and the system is equivalent to a linear system on the global attractor. In this case, we show that each omega limit set has the structure of an $n$-torus for some integer $n$. This differs from the situation of polyhedral norms for analytic vector fields. In this case the omega limit sets are always single points. In $\mathbb{R}^2$, weighted $l^2$ norms are the only norms for which a nonexpansive system (with respect to a weighted $l^2$ norm) does not necessarily converge to the equilibrium set. At the end we describe some examples.

\section{Background}

In the following we will describe norms and dynamical systems with special properties relating to the norm. We will assume that we have an autonomous system $\dot{x} = f(x)$ where $x \in \mathbb{R}^n$ and $f(\cdot )$ is $C^1$. We assume that we are given a particular norm $\|.\|$ on $\mathbb{R}^n$. We will define the forward time evolution of the system $\dot{x} = f(x)$ to be $\phi_t$. We assume that $\phi_t$ is defined for all $t \geq 0$. Given a vector field $f(x)$, we let the Jacobian evaluated at a point $x$ be $\mathcal{J}_f(x)$.

Now we will recall a few basic definitions that will be used in the sequel. A \textit{state space} $X$ is a forward invariant set for the system. An \textit{$\omega$ limit set} of a point $x$ is the set of points  $\cap_{t \geq 0} \overline{\cup_{s \geq t} \phi_s(x)} $. A system $\dot{x} = f(x)$ is \textit{nonexpansive} with respect to a norm $\| . \|$ if for all $t > 0$ and all $x,y$ in the system's state space we have that $\| \phi_t(x) - \phi_t(y) \| \leq \|x - y\|$.  
In other words, the flow maps are Lipschitz with Lipschitz constant $1$.
The \textit{global attractor} of a system $\dot{x} = f(x)$ (relative to the state space $X$) is the set ${\gattract}  = \cap_{t\geq 0 } \phi_t(X)$. An \textit{isometry} of a normed vector space $V$ is a mapping $F: V \rightarrow V$ such that $\|x-y\| = \|F(x) - F(y)\|$ for all $x,y \in V$. A \textit{discrete subgroup} of $GL_n$ is a group $G \subseteq GL_n$ such that for every $g \in G$ there exists an open ball $O_g$ such that $O_g \cap G = g$. For any positive integer $n$ we indicate the \textit{n-torus} by the $n$ product $S^1 \times S^1 \times... \times S^1 = (S^1)^n$ where $S^1$ is the circle.

\section{Some results on nonexpansivity}

While our results apply to non-compact state spaces, we can motivate working on compact state spaces by using Corollary~\ref{cor:equilibrium} below, which says that, under minimal assumptions, we can restrict analysis to a sufficiently large compact ball which contains the initial conditions of interest. The result will follow from Proposition~\ref{pro: unbounded or has fixed point}.
First we recall a well-known lemma that extends Brouwer's fixed point theorem to flows (Yorke’s fixed point theorem in \cite{FB-CTDS}); for completeness, we provide a self-contained proof.

\begin{lem}
\label{lem:fixed point}
    Suppose we have a $C^1$ vector field $f(x)$ and a compact and convex forward invariant set $X$. Then $f(x)$ has an equilibrium on $X$.
\end{lem}

\begin{proof}
    Suppose $f(x)$ did not have an equilibrium on $X$. Then for any point $p$ there exists a time $t_p$ such that for $t < t_p$ that $\phi_{t}(p) \neq p$, i.e., $p$ is taken to a different point. In fact, this $t_p$ also works for all points in a neighborhood of $p$ (this can be seen via the Flow-box Theorem). Cover $X$ with all such neighborhoods. Since $X$ is compact we can then pick finitely many of these neighborhoods to cover $x$. Then we can take $t_f$ to be the minimum of all the times corresponding to these neighborhoods. Thus for $t < t_f$ we have that $\phi_{t}$ has no fixed points, contradicting Brouwer's fixed point theorem. Thus $f$ must have a fixed point.
\end{proof}

The following result, and the main ideas of its proof, are given as Theorem 19 in \cite{9403888}. We provide a streamlined proof for completeness.

\begin{proposition}\label{pro: unbounded or has fixed point}
    Suppose we have a nonexpansive time-invariant system $\dot{x} = f(x)$. Then exactly one of the following two conditions hold: 
    \begin{enumerate}
        \item Every trajectory of the system is unbounded.
        \item The system has an equilibrium point $x^*$ (that is $f(x^*) = 0$), and every trajectory is bounded.
    \end{enumerate}

    \begin{proof}
        First assume we have a bounded trajectory with initial point $x$. Thus we can consider $\omega(x)$, the omega limit set of $x$, which is a nonempty backward and forward invariant compact set for the system. For arbitrary $\epsilon > 0$ define $B_{\epsilon}(p) = \{y \in \mathbb{R}^n | \|y - p \| \leq \epsilon \}$. Consider $C_{\epsilon} = \cap_{p \in \omega(x)} B_{\epsilon}(p)$. Note that $C_{\epsilon}$ is convex and compact, since it is the intersection of convex and compact sets. If $y \in C_{\epsilon}$ then we must have for all $p \in \omega(x)$ that $\|y - p\| \leq \epsilon$, due to the definition of $C_{\epsilon}$. Fix an arbitrary $t \geq 0$ and $p \in \omega(x)$. Since $\omega(x)$ is backward invariant there exists $p' = \phi_{-t}(p) \in \omega(x)$ such that $\phi_t(p') = p$. Since the system is nonexpansive we must have that 
        \[
         \|\phi_t(y) - \phi_t(p') \|  \leq \|y - p' \| \leq \epsilon.
        \]
        From this we have that
        \[  \| \phi_t(y) - p \| =\|\phi_t(y) - \phi_t(p') \| \leq \epsilon. 
        \]
        
        Since $p$ was arbitrary, we must have that $\phi_t(y) \in C_{\epsilon}$ for all $t\geq 0$, and so $C_{\epsilon}$ is forward invariant for all $\epsilon \geq 0$ (if it is empty the statement is trivial).

        Now we can pick $\epsilon$ large enough such that $C_{\epsilon}$ is nonempty (which is clearly possible since $\omega(x)$ is compact). Then we can apply Lemma \ref{lem:fixed point} to conclude that $\dot{x} = f(x)$ has a fixed point in $C_{\epsilon}$.
    \end{proof}

\end{proposition}

 The following corollary follows immediately from Proposition \ref{pro: unbounded or has fixed point}. 

\begin{cor}\label{cor:equilibrium}
    Suppose we have a system that is nonexpansive with respect to a norm $\|.\|$ and that has a precompact trajectory. Then the system has at least one equilibrium point $p$, and every norm ball $B_{p,d} = \{ x \in \mathbb{R}^n | \|p - x \| \leq d \}$ is a compact forward invariant set.
\end{cor}

Thus, for the remainder of this section, we will make the assumption, restricting if necessary to balls around an equilibrium, that \textit{all state spaces $X$ we consider are compact}.

\subsection{Compact state space}

\begin{lem}
\label{lem:time varying uniform convergence}
    Suppose we have a dynamical system $\dot{x} = f(x)$ with a $C^1$ vector field $f(x)$ and a compact forward invariant state space $X$. Then for any $\epsilon > 0$ there exists $T > 0$ such that for all $t,s > T$ and for all $x,y \in X$ we have that $| \|\phi_t(x) - \phi_t(y) \| - \| \phi_s(x) - \phi_s(y) \| | < \epsilon.$

\end{lem}

\begin{proof}
    For the following let $n \in \mathbb{Z}$. Consider the sequence of functions $d_n: X \times X \rightarrow \mathbb{R}_{\geq 0}$ for $n \geq 0$ (here we give $X \times X$ the sup product metric) defined by
    \[
        d_n(x,y) = \|\phi_n(x) - \phi_n(y) \|.
    \]
    We have that $X \times X$ is a compact set. Note that $d_n$ satisfies the triangle inequality (since the norm satisfies it) and is symmetric. Due to the nonexpansivity of $\phi_n$ we have that $d_n(x,y)$ is monotonically decreasing in $n$ for any given $(x,y)$. Due to the nonnegativity of norms we also have that $d_n \geq 0$ for each $n$ and so $d_n$ is bounded below. Thus as $n \rightarrow \infty$ we have pointwise convergence to some function $d$. Note that $d_n$ is also a continuous function for each $n$. This is due to $\phi_n$ and the norm function both being continuous. 
    
    Lastly, we note that $d(x,y)$ is a continuous function on $X \times X$. Indeed, by the triangle inequality we have that for all $x,y,x',y' \in X $:
    \begin{align*}
        &d_n(x',y') - d_n(x',x) - d_n(y',y)  \leq d_n(x,y) \\ &d_n(x,y) \leq d_n(x',x) + d_n(x',y') + d_n(y',y).
    \end{align*}
    Suppose that $(x,y),(x',y')$ are close to each other in $X \times X$, i.e., $\max\{\|x - x'\|, \|y - y' \| \} < \epsilon$. Due to nonexpansivity, we have that $\max\{\|\phi_n(x) - \phi_n(x')\|, \|\phi_n(y) - \phi_n(y') \| \} < \epsilon$ for all $n \geq 0$. Thus we have that $0 \leq d_n(x,x')< \epsilon$ and $0 \leq d_n(y,y') < \epsilon$ for all $n \geq 0$. Using these bounds in the previous inequality, we now have that 
    \[
      d_n(x',y') - 2 \epsilon  < d_n(x,y) < d_n(x',y') + 2 \epsilon.
    \]
    Thus we have that $|d_n(x,y) - d_n(x',y')| < 2 \epsilon$ for all $n \geq 0$. Taking the limit in $n$, we see that $|d(x,y) - d(x',y')| \leq 2 \epsilon$. Thus the function $d$ is continuous.

    Now we can apply Dini's theorem and so we have that $d_n$ in fact converges uniformly to $d$. Thus there exists $N$ such that for $n \geq N$ we have that $d_n(x,y) - d(x,y) < \epsilon$ for all $x,y \in X$. Thus we also have that $d_n(x,y) - d_{n+k}(x,y) < \epsilon$ for all $n \geq N$ and all integers $k \geq 0$ (i.e., this sequence is a Cauchy sequence at each point $(x,y)$).
\end{proof}

Note this lemma says that points in the state space uniformly approach their minimum distance from each other. We then have the following:

\begin{cor}
\label{lem:isometry on attractor}
    Suppose we have a $C^1$ system $\dot{x} = f(x)$ with compact forward invariant state space $X$. Then for any real number $t \geq 0$ the time evolution operator $\phi_t$ is an isometry on the set ${\gattract}  = \cap_{t\geq 0 } \phi_t(X)$ (i.e., the global attractor of the system).
\end{cor}

\begin{proof}
   Defining $d_n(x,y)$ and $d(x,y)$ as in Lemma \ref{lem:time varying uniform convergence}, we know that for any $\epsilon>0$ there is an integer $N>0$ so that $d_n(x',y') - d_{n+k}(x',y') < \epsilon$ for all $x',y'\in {\gattract}$ and all $n>N$ and $k>0$. Pick now any $x,y\in {\gattract}$ and any two integers $n>0$ and $k>0$ such that $n > N$. Since ${\gattract} \subseteq \phi_n(X)$, we have that there exist $x',y'$ such that $\phi_n(x') = x$ and $\phi_n(y') = y$.  We have that
    \begin{align*}
    & \|x - y \| - \|\phi_k(x) - \phi_k(y) \|  \\
    &=  \|\phi_{n}(x') - \phi_{n}(y') \| - \|\phi_{k+n}(x') - \phi_{k+n}(y') \|  \\
    &= d_n(x',y') - d_{n+k}(x',y') < \epsilon.
    \end{align*}
    Since $\epsilon$ can be arbitrarily small, we have that  $ \|x - y \| - \|\phi_k(x) - \phi_k(y) \| =  0$, or $ \|\phi_k(x) - \phi_k(y) \| = \|x - y \|$. Since $k$ was arbitrary, this holds for all integers $k \geq 0$. Note that this implies, for example, for each $0 \leq t \leq 1$ that (by nonexpansivity) 
    \[
        \|\phi_0(x) - \phi_0(y) \| \geq \|\phi_t(x) - \phi_t(y) \| \geq \|\phi_1(x) - \phi_1(y) \|.
    \]
    Since the left and right terms are equal, all the inequalities are in fact equalities. The same argument can be applied to any positive real number $t$.
    
    Thus for $x,y \in {\gattract}$ and any real number $t \geq 0 $ we have that $ \|\phi_t(x) - \phi_t(y) \| = \|x - y \|$, and so the time evolution operator is an isometry on this set.
\end{proof}

Observe that ${\gattract}$ is nonempty, since it is an intersection of a decreasing family of compact sets.
A key property is that every trajectory converges to ${\gattract}$, as shown next.

\begin{lem}
    Every omega limit set is contained in ${\gattract}$.
\end{lem}

\begin{proof}
    Take an arbitrary $x \in X$. Since the state space $X$ is compact, the solution starting from $x$ has a nonempty compact, connected, and backward and forward invariant omega limit set $\omega(x)$, and the solution converges to it. Pick any $y \in \omega(x)$. Then for all $t>0$ we have that $\phi_{-t}(y) \in \omega(x) \subseteq X$ and so $y \in \phi_t(X)$. Thus $y \in \cap_{t >0} \phi_t(X) = {\gattract}$. Since $y$ was arbitary, this shows that $\omega(x) \subseteq {\gattract}$.
\end{proof}

One could also derive Corollary~\ref{lem:isometry on attractor} by appealing to a result from Freudenthal and Hurewicz \cite{Freudenthal1936} which showed that every nonexpansive map from a totally bounded metric space (for example, any compact space) onto itself must be an isometry; see also \cite{Ding2011}.

\subsection{Strictly convex norms}

Recall that a norm $\|.\|$ is \textit{strictly convex} if and only if whenever $x$ and $y$ are two distinct points with $\|x\| = r$ and $\|y\| = r$ for some $r > 0$, we have that for $0 <\alpha < 1$ then $\|\alpha x + (1-\alpha) y\| < r$. For the case where a given norm is strictly convex we have the following uniqueness lemma:

\begin{lem}
\label{lem:strictly convex point unique}
    Suppose we have a strictly convex norm $\| .\|$. Pick two points $x,y$ and any number $0\leq a < \|x-y\|$. Then the point $z$ that satisfies $\|x - z\| = a < \|x-y\|$ and $\|x-y\| = \|x - z\| + \|z - y\|$ exists and is unique.
\end{lem}

\begin{proof}
Note there exists such a point, since we can simply take $z = (1-\frac{a}{\|x-y\|}) x + \frac{a}{\|x-y\|} y$.

If there were two points $z$ and $z'$ with the claimed property, consider $x - z$ and $x - z'$. Pick any number $\alpha$ such that $0 < \alpha < 1$. Let $z'' = \alpha z + (1 - \alpha) z'$. Note we have that
\[
x - z'' = \alpha (x - z) + (1-\alpha) (x-z')
\]
and
\[y - z'' = \alpha (y - z) + (1-\alpha) (y-z') \,.
\]
By the triangle inequality we have that 
\[
\|x-z''\| + \|z'' -y\| \geq \|x-y\|.
\]
Note that $\|x - z \| = \|x - z' \| = a$ and $\|y - z\| = \| y - z'\| = \|x - y \| - a$. By strict convexity we have that
\[
\|x - z''\| = \| \alpha (x - z) + (1-\alpha) (x-z') \| < a
\]
and
\[
\|y - z''\| = \| \alpha (y - z) + (1-\alpha) (y-z') \| < \|x-y\| - a \,.
\]
This gives us 
\[
\|x-z''\| + \|z'' -y\| < \|x-y\|.
\]
This contradicts the triangle inequality, and thus the point is unique.
\end{proof}

Notice that Lemma~\ref{lem:strictly convex point unique} need not hold for non-strictly convex norms. For example, consider the $\ell^1$ norm and $x=(0,0)$, $y=(1,1)$. Then with $a=1/2$ we can pick $z_1=(0,1)$ and $z_2=(1,0)$ to satisfy the property that $\|x-y\| = 2 = 1+1 = \|x - z\| + \|z - y\|$.

From now on in this section, we assume that the norm being considered is strictly convex.

\begin{lem}
\label{lem:convex sums}
    For $x,y \in {\gattract}$, $t \geq 0$ and $1 \geq \lambda \geq 0$ we have that $\phi_t(\lambda x + (1 - \lambda) y) = \lambda \phi_t(x) + (1-\lambda) \phi_t(y)$
\end{lem}

\begin{proof}
    Let $d(x,y) = \|x- y\|$.  Let $z = \lambda x + (1 - \lambda) y$, $ d(z,x) = a$ and $d(z,y) = b$. We have that $d(x,y) = d(z,y) + d(z,x) = a+b$. Note that $z$ is the unique point (due to Lemma \ref{lem:strictly convex point unique}) such that $d(z,x)$ and $d(z,y)$ take on these real values $a$ and $b$, respectively. 

    By Corollary~\ref{lem:isometry on attractor}, we have that
    $ d(\phi_t(x),\phi_t(y)) = d(x,y) = a+b$. 
    Since $z$ might not be in ${\gattract}$, we cannot yet assert the isometric relationships $d(\phi_t(z),\phi_t(x)) = a$ or $d(\phi_t(z),\phi_t(y)) = b$.
    However, by nonexpansivity we have that $d(\phi_t(z), \phi_t(x)) \leq d(z, x) = a$, and $d(\phi_t(z), \phi_t(y)) \leq d(z,y) = b$. By the triangle inequality we have that 
    \begin{align*}
        a + b &= d(\phi_t(x),\phi_t(y))) \\
        &\leq d(\phi_t(z), \phi_t(x)) + d(\phi_t(z), \phi_t(y)) \\
        &\leq a + b.
    \end{align*}
    Since the left and right hand are the same we must have that $d(\phi_t(z),\phi_t(x)) = a$ and $ d(\phi_t(z),\phi_t(y)) = b$, as desired.

    Thus $\phi_t(z)$ satisfies the conditions in Lemma \ref{lem:strictly convex point unique} where $x$ and $y$ are replaced with $\phi_t(x)$ and $\phi_t(y)$, respectively. This implies $\phi_t(z) = \lambda \phi_t(x) + (1-\lambda) \phi_t(y)$. Indeed, we have that $d(\phi_t(z),\phi_t(x)) = \| \phi_t(z) - \phi_t(x)\| = (1 - \lambda) \|\phi_t(y) - \phi_t(x)\| = (1 - \lambda) \|y - x\| = a $ and similarly we have that $d(\phi_t(z),\phi_t(y)) = b$. 
\end{proof}

\begin{lem} 
\label{lem:A backward invariant}
    The set ${\gattract}$ is backward and forward invariant.
\end{lem}

\begin{proof}
    First observe that, for any $s > t$, $\phi_s(X) \subseteq \phi_t(X)$. Indeed, if $x \in \phi_s(X)$ then $x=\phi_s(z) = \phi_{t}(\phi_{s-t}(z))$ for some $z\in X$. Thus $x=\phi_t(y)$, with $y:=\phi_{s-t}(z)$.
   
    Now recall ${\gattract} = \cap_{t>0} \phi_t(X)$. By our previous observation for any $s > 0$ we also have that ${\gattract} = \cap_{t > s} \phi_t(X)$. We have that $x \in {\gattract}$ iff $x \in \phi_t(X)$ for all $t > 0$ iff for any $s>0$ we have that $\phi_s(x) \in \phi_t(X)$ for $t > s$ iff $\phi_s(x) \in \cap_{t > s} \phi_t(X) = {\gattract}$. Thus ${\gattract}$ is forward invariant.

    Suppose again we have $x \in {\gattract}$. Now for each $s >0$ we have that $x \in {\gattract}$ iff $x \in \cap_{t > s} \phi_t(X)$ iff $\phi_{-s}(x) \in  \cap_{t > 0} \phi_t(X) = {\gattract}$. Since $s$ was arbitrary ${\gattract}$ must be backwards invariant as well. 
\end{proof}

\begin{lem}
    The set ${\gattract}$ is convex.
\end{lem}

\begin{proof}
    Take arbitrary $x,y \in {\gattract}$. 
    Pick any $0 < \lambda < 1$.
    We need to show that $z:=\lambda x + (1-\lambda) y \in {\gattract}$. 
    Since ${\gattract}$ is backwards invariant by Lemma \ref{lem:A backward invariant}, for all $t>0$ there exist $x',y' \in {\gattract}$ such that $\phi_t(x') = x$ and $\phi_t(y') = y$. By Lemma \ref{lem:convex sums} this means that, for each $0 < \lambda < 1$: 
    \[
    \phi_t(\lambda x' + (1 - \lambda) y') = \lambda \phi_t(x') + (1-\lambda) \phi_t(y') = \lambda x + (1-\lambda) y .
    \]
    The above equation implies that for all $t>0$, we can find a $z'$ such that $\phi_t(z') = z$. Thus for all $t>0$ we must have that $z \in \phi_t({\gattract}) \subseteq \phi_t(X)$. Thus $z \in \cap_{t>0} \phi_t(X) = {\gattract}$. In other words, for all $x,y \in {\gattract}$ and for all $0 < \lambda < 1$ we have that $\lambda x + (1-\lambda) y \in {\gattract}$, as claimed.
\end{proof}

Since ${\gattract}$ is compact and convex, the vector field $f$ restricted to ${\gattract}$ has a fixed point, by Lemma \ref{lem:fixed point}. Without loss of generality, we can view this fixed point as the origin in $\mathbb{R}$, so from now on we assume that ${\gattract}$ contains $0$ and that $f(0)=0$. Thus also $\phi_t(0)=0$ for all $t$.

In the next result, we use Mankiewicz's Theorem
(see for example \cite{Jung+2022+1353+1379}).
This theorem applies to any isometry $g : E  \rightarrow Y $,
where $E$ is a nonempty subset of a real normed space $X$, and $Y$ is a 
real normed space.
If either both $E$ and $g(E)$ are convex bodies (compact and convex with nonempty interior) or if $E$ is open and connected and $g(E)$ is open, then $g$ can be uniquely extended to an affine isometry $F : X \rightarrow Y$.

\begin{lem}
\label{lem:one parameter isometry family}
    Let $V$ be the linear span of ${\gattract}$. There exists a one-parameter family of affine isometries $F_t$ on $V$ such that $F_t$ is an extension of $\phi_t$ restricted to ${\gattract}$. 
\end{lem}

\begin{proof}
    Fix any $t>0$. We know by Lemma \ref{lem:isometry on attractor} that $\phi_t$ is an isometry on the convex set ${\gattract}$.
    If ${\gattract}=\{0\}$ then the result is trivial, so assume ${\gattract}\not=\{0\}$. Let $\{v_1, \ldots, v_m\}$ be a maximal linearly independent set of vectors in ${\gattract}$. Thus $V$ is the span of $\{v_1, \ldots, v_m\}$. 
    Every linear combination $p = \sum_{i=1}^m\lambda_iv_i$ with all $\lambda_i>0$ and $\sum_{i=1}^m\lambda_i<1$ belongs to ${\gattract}$ (since $p=(1-\sum_{i=1}^m\lambda_i)0 + \sum_{i=1}^m\lambda_iv_i$ is in ${\gattract}$, by convexity and because $0\in {\gattract}$). So ${\gattract}$ has a nonempty interior in $V$. It follows that ${\gattract}$ is a convex body relative to $V$.     
    We now apply Mankiewicz's Theorem with $g=\phi_t$, $E={\gattract}$, and $X=Y=V$. Note that $g({\gattract})={\gattract}$ because ${\gattract}$ is backwards complete, so that $g({\gattract})$ is a convex body as needed for the theorem.
    Thus we have an extension to an affine transformation $F_t$ on $V$. 
\end{proof}

As every $\phi_t$ vanishes at zero (recall that we assumed this without loss of generality), so do the mappings $F_t$ from Lemma \ref{lem:one parameter isometry family}. Therefore each $F_t$ is a linear map.
Since each $F_t$ is an isometry, $F_t$ is nonsingular, that is, $F_t\in GL_m(\mathbb{R})$.

\begin{lem}
\label{lem:Ft continous}
    The mappings $F_t$ vary continuously with $t$.
\end{lem}

\begin{proof}
    Since $f$ is a $C^1$ vector field, the $\phi_t$ mappings vary continuously with $t$ on the compact and convex set ${\gattract}$. 
    Suppose that $V$ (i.e., the span of ${\gattract}$) is $m$ dimensional.
    We can find $m$ linearly independent vectors $x_1,x_2,...,x_m$ in ${\gattract}$ that span $V$.
    Since ${\gattract}$ is forward and backwards invariant, for each $1 \leq i \leq m$ and each $t$, $\phi_t(x_i)\in {\gattract}$, and hence $F_t(x_i) = \phi_t(x_i)$. Thus $F_t(x_i)$ varies continuously with $t$ since $\phi_t(x_i)$ varies continuously with $t$. We conclude that the mapping $t \rightarrow F_t$ is continuous as a map $\mathbb{R} \rightarrow GL_m(\mathbb{R})$. 
\end{proof}

\begin{lem}
    We have that $F_t = e^{Bt}$ for some linear transformation $B$ on $V$.
\end{lem}

\begin{proof}
    Since $F_0 = I$ (here $I$ is the identity transformation), $F_t F_s = F_{s+t}$, and $F_t$ varies continuously in $t$ , the set of transformations $F_t$ is a one parameter subgroup of $GL_m(\mathbb{R})$. By Theorem 2.14 in \cite{hall2003lie} we can conclude that there exists a unique linear map $B \subseteq GL_m(\mathbb{C})$ such that $F_t = e^{Bt}$. Note that since $B = \frac{d}{dt} F_t|_{t=0}$ we in fact must have $B \subseteq GL_m(\mathbb{R})$
\end{proof}

The following is a standard property of center manifolds of linear time-invariant systems (see for example Problem 5 in Problem Set 9 in~\cite{perko}); we provide a proof for completeness.

\begin{lem}
\label{lem:linear ellipsoid trajectories}
    Suppose a linear system $\dot{x} = Bx$ satisfies that its trajectories are bounded and do not converge to 0. Then the matrix $B$ has only eigenvalues with 0 real part, and it is diagonalizable.
\end{lem}

\begin{proof}
    Note that if any eigenvalue had negative real part, then we can find a trajectory converging to 0. If any had positive real part, we could find a trajectory diverging to infinity.
    
    Note that there exists $P \in GL_m(\mathbb{C})$ such that $B = P N P^{-1}$ where $N$ is in Jordan normal form. Note that $e^{Nt}$ has diagonal blocks with $t$'s on the off diagonal if any of the blocks are not diagonal matrices. This would imply again diverging trajectories, thus all the blocks must be diagonal and so $B$ is diagonalizable.
\end{proof}

We will call such linear differential equations \textit{conserved} linear equations.
A quadratic Lyapunov function for such systems can be constructed as usual through the solution of a Lyapunov equation (see e.g.~\cite{mct}). Again for completeness, we provide a proof.

\begin{lem}
\label{lem:conserved P form}
    Every conserved linear system has a quadratic form $P$ such that $\frac{d x^\top P x}{dt} = 0$. 
\end{lem}

\begin{proof}

Consider a conserved linear system $\dot{x} = Bx$. Note by Lemma \ref{lem:linear ellipsoid trajectories} we can diagnoalize $B$ with a real matrix $L$. In other words, $L^{-1} B L$ is such that it is a skew symmetric matrix consisting of diagonal blocks of the form 
\[
\begin{bmatrix}
    0 & \alpha \\
    -\alpha & 0
\end{bmatrix}.
\]
Let $P = L^\top L$. We have that
\begin{align*}
    B^\top L^\top L + L^\top L B &= (L^{-1} B L)^\top L^\top L + L^\top L (L^{-1} B L) \\
    &= L^\top B^\top L + L^\top B L \\
    &= -L^\top B L + L^\top B L = 0.
\end{align*}
Thus $\frac{d x^\top P x}{dt} = x^\top(B^\top P + P B)x = 0$.
\end{proof}

\begin{lem}
\label{lem:point in own omega}
    For a conserved linear system $\dot{x} = Bx$, every point $x_0$ is in its own omega limit set.
\end{lem}

\begin{proof}
    Assume upon a linear transformation that $B$ is block-diagonal with blocks that are either 2 by 2 skew symmetric matrices or 1 by 1 zero matrices. The trajectory of $x_0$ is thus $e^{Bt} x_0$ where $e^{Bt}$ consists of 2 by 2 blocks of rotation matrices on its diagonal of the form.
\[
    \begin{bmatrix}
        \cos (\alpha_i t) & -\sin (\alpha_i t) \\
         \sin(\alpha_i t) & \cos (\alpha_i t)
    \end{bmatrix}
    \]
    as well as 1's in diagonal entries corresponding to zero entries in $B$.
    Put the $\alpha_i$ terms into a row vector $[\alpha_1,\alpha_2,...,\alpha_l]t$ and consider this vector modulo $2 \pi$. Divide up the region $[0, 2\pi]^n$ into boxes of side length at most $\epsilon$. Note that for any $\epsilon > 0$ and any $\delta > 0$ by the pigeonhole principle we can always find $t_1$ and $t_2$ such that $|t_1 - t_2|$ is bounded below by $\delta > 0$ and that
    \[|[\alpha_1,\alpha_2,...,\alpha_l]t_1 - [\alpha_1,\alpha_2,...,\alpha_l]t_2| < [\epsilon,\epsilon,...,\epsilon].\]
(Here the absolute value and comparison are done element-wise.) 
    Indeed, the set of points $\{(t_1 + \delta j) [\alpha_1,\alpha_2,...,\alpha_l]| j \in \mathbb{N} \}$ (taken modulo $2 \pi$) is an infinite set of points in $[0, 2\pi]^n$ and thus we can find 2 different points in the same box (from the boxes we have previously divided our region into). These two points precisely satisfy out inequality. 

    Thus if $t_2 > t_1$ then at time $t = t_2-t_1$ we have that $e^{Bt}$ is close to the identity matrix. This is due to the fact that if all the $\alpha_i t$ are close to multiples of $2 \pi$, all the 2 by 2 rotation matrices will be close to being identity matrices. Picking $\delta_i = i$ and $\epsilon_i = \frac{1}{i}$ we can always find a corresponding $t_i > \delta_i$ such that
    $e^{B t_i}  \rightarrow I$ as $t_i \rightarrow \infty$. In particular, for each $x_0$,
    $e^{B t_i} x_0 \rightarrow x_0$ as $t_i \rightarrow \infty$. Thus $x_0$ is in its own omega limit set.
\end{proof}

\begin{lem}
\label{lem:describe linear system trajectories}
    For a conserved linear system $\dot{x} = Bx$ the trajectories are homeomorphic to an $k$-torus $(S^1)^k$ for some integer $k$.
\end{lem}

\begin{proof}

Via a linear transformation, we can assume that $B$ consists of 2 by 2 blocks of skew symmetric matrices on its diagonal, and 0's elsewhere. Our trajectories are always of the form $\{ e^{Bt} x_0| t \geq 0 \}$ for some initial point $x_0$. Whenever we have two entries of $x_0$ equal to 0, and they both correspond to the same block, remove this block from $e^{Bt}$ (otherwise, these entries would simply remain 0 for the entire trajectory). Going forward we consider this reduced form of $e^{Bt}$.

Let $T$ be the set of matrices which consist of 2 by 2 rotation matrices on the diagonal, 1's elsewhere on the diagonal, and 0's off the diagonal (i.e., the same general structure as $e^{Bt}$), seen as a Lie subgroup of an appropriate $GL(k,\R)$.  It is easy to see that $T$ is a compact, connected, and commutative Lie group, and thus it is a torus (Theorem 11.2 in \cite{hall2003lie}).
Note that the closure of $G = \{e^{Bt} | t \in \mathbb{R} \}$, call it $\bar{G},$ is a subgroup (the closure of a subgroup is still a subgroup) of $T$. Since $\bar{G}$ is a closed subgroup of $T$, it must be compact and commutative. It is also a Lie subgroup of $T$ by the Closed Subgroup Theorem (see Theorem 20.12 in \cite{lee2012introduction}). Since $G$ is connected so is $\bar{G}$. Thus $\bar{G}$ is a compact, connected and commutative Lie subgroup of $T$ and therefore it must be a torus itself. 

Define $L = \{e^{Bt}x_0| t \in \mathbb{R}\}$. By Lemma \ref{lem:point in own omega} we have that $x_0$ is in $\omega(x_0)$ and since omega limit sets are backward and forward invariant we must also have that $L \subseteq \omega(x_0)$ and thus since omega limit sets are closed sets that $\bar{L} \subseteq \omega(x_0)$. Thus since $\omega(x_0) \subseteq \bar{L}$ we have that our omega limit set is in fact precisely $\bar{L}$.

Thus we can think of $\bar{G}$ as acting on $x_0$, and since the stabilizer is trivial we have $\bar{L}$ is diffeomorphic to $\bar{G}$ (see Theorem 21.18 in \cite{lee2012introduction}). Thus the omega limit set of $e^{Bt} x_0$ must also be a torus.
\end{proof}

We are now ready to prove our main result:

\begin{thm}
\label{thm:strictly convex describe omega limit sets}

Suppose we have a dynamical system $\dot{x} = f(x)$ which is nonexpansive for a strictly convex norm $\|.\|$. Suppose it has at least one bounded trajectory. Then all the trajectories are bounded, and their omega limit sets are that of some fixed conserved linear system $\dot{x} = Bx$. In particular, the omega limit sets are homeomoprhic to $(S^1)^k$ for some integer $k$.
\end{thm}

\begin{proof}
    By Lemma \ref{lem:linear ellipsoid trajectories} we have that on the set ${\gattract}$ our dynamics must be equivalent (up to translation) to that of a linear system. By Lemma \ref{lem:describe linear system trajectories} we have that all the omega limit sets are tori.
\end{proof}

\subsection{Nonexpansive polyhedral norms}

We provide a self-contained proof that for \emph{(real-)analytic} vector fields which are nonexpansive with respect to a norm, we have a stronger convergence result. The following is essentially Theorem 21 from \cite{jafarpour2020weak}, but certain technical details were missing in the proof in that paper.

\begin{thm} \cite{jafarpour2020weak}
    Suppose we have a system $\dot{x} = f(x)$ where $f(x)$ is analytic and has bounded trajectories. Suppose the system is nonexpansive with respect to some polyhedral norm. Then the system converges to its equilibria set.
\end{thm}

\begin{proof}

One can show that $\| f(x(t))\|$ is nonincreasing along any trajectory, 
because $(d/dt){f(x(t))} = \mathcal{J}_f(t) f(x(t))$ and the logarithmic norm of $\mathcal{J}_f(t)$ is nonpositive.
It follows by the LaSalle's Invariance Principle that every solution approaches a set $Z_c:= \{x_0 \,|\, \| f(\phi_t(x_0))\|\equiv c\}$ for some $c\geq0$.

We claim that any such set $Z_c$ consists solely of equilibria.
Pick any point $x_0\in Z$ and the corresponding trajectory $x(t)=\phi_t(x_0)$.
By definition of $Z_c$, $x(t)\in Z_c$ for all $t\geq0$.
We claim that $c=0$, i.e. that $x(t)\equiv x_0$, so that $x_0$ is an equilibrium.
Indeed, suppose that $c \not =0$.
Then $f(x(t))$ is always a point on a norm ball of a constant (nonzero) size. Thus it must spend a finite time on a face of this ball of constant size. Suppose that this face has normal vector $\eta$. Then $\eta \cdot f(x(t))$ will be a constant value, for a set of times in a set of positive measure, and so must be a constant value for all time, by analyticity. This implies that $\eta \cdot f(x(t))$ has this constant value for all $t\geq0$, forcing the velocity vector $f(x(t))$ to always point in a certain direction (i.e., along the direction of $\eta$), forcing the trajectory to be unbounded, a contradiction.
\end{proof}

\subsection{Nonexpansive maps on \texorpdfstring{$\mathbb{R}^2$}{Lg}}

In the special case of $\mathbb{R}^2$, we have some stronger results. In the following, we do not assume the norm on $\mathbb{R}^2$ is strictly convex.

\begin{lem}
\label{lem:l^2 norms preserve 1 param}
    The only norms preserved by a nontrivial one parameter family of linear isometries of the form $e^{Bt}$ are the weighted $l^2$ norms.
\end{lem}

\begin{proof}
We consider all possible bounded trajectories of the form $e^{Bt}x_0$. This corresponds to trajectories of the linear system $\dot{x} = Bx$. Note that all the eigenvalues of $B$ must have 0 real part, otherwise we would have points converging to 0 or diverging to $\infty$, contradicting that $e^{Bt}$ should be an isometry for all $t$.

Note that by Lemma \ref{lem:conserved P form} that there exists a matrix $P$ such that $\frac{d (x^\top P x)}{dt} = 0$. Note that $x^\top Px = 1$ is the unit ball of this norm which is preserved by the vector field $\dot{x} = Bx$, and so this preserved norm is unique up to multiplication by a scalar.
\end{proof}

\begin{lem}
\label{lem:no limit cycles}
    If a global attractor ${\gattract}$ contains a limit cycle, the only norm we can preserve on ${\gattract}$ is a weighted $l^2$ norm.
\end{lem}

\begin{proof}
    Suppose we have a limit cycle, and let $I$ be the limit cycle with its interior. Then since $\phi_t$ takes $I$ to itself for all time, so $\phi_t$ must be an isometry on $I$ by Lemma \ref{lem:isometry on attractor}. It contains an open set so by Mankiewicz's Theorem the map $\phi_t$ restricted to the interior of $I$ can be extended to an affine map $F_t$. Thus by Lemma \ref{lem:l^2 norms preserve 1 param} the preserved norm must be a weighted $l^2$ norm.
\end{proof}

\begin{lem}
\label{lem:converge to fixed point}
    If an $\omega$ limit set of a point $p$ contains an equilibrium point, then $p$ converges to that limit point.
\end{lem}

\begin{proof}
    If $p$ is an equilibrium point in the $\omega$ limit set of a point $x$, then for every $\epsilon > 0$ we can find $t > 0$ such that $\| p - \phi_t(x) \| < \epsilon$. Since the system under consideration is nonexpansive, we have that $\| p - \phi_t(x) \| < \epsilon$ for all $t > T$. Thus the trajectory is simply converging to $p$.
\end{proof}

\begin{lem}
    If a system is nonexpansive for a norm which is not a weighted $l^2$ norm, then all bounded trajectories must converge to the equilibria set.
\end{lem}

\begin{proof}

    Suppose the system is nonexpansive for a norm which is not a weighted $l^2$ norm. By Lemma \ref{lem:no limit cycles} the system cannot have any limit cycles. By the Poincare-Bendixson theorem any $\omega$ limit set that is not a limit cycle must contain a fixed point, but by Lemma \ref{lem:converge to fixed point} the fixed point it the only fixed point and we must converge to it.

\end{proof}

\section{A necessary and sufficient condition for nonexpansivity}

Here we will provide a necessary and sufficient description of nonexpansivity with respect to a norm. This condition is connected to the supporting hyperplanes of a unit ball of said norm. This can be seen as a type of Demidovich condition for contractivity \cite{9799744}.

Let $B_d = \{x \in \mathbb{R}^n | \|x\| \leq d\}$. For all $v \in \mathbb{R}^n$ let $N_v$ be the set of all possible normal vectors of hyperplanes that support $B_{\|v\|}$ at $v$ and are orientated toward the complement of $B_{\| v\|}$. In the following let $\mathbb{X} = \mathbb{R}^n$. 

\begin{thm}
\label{thm:nonexpansive demidovich}
    Suppose we have a dynamical system $\dot{x} = f(x)$ and a norm $\|.\|$ on the state space $\mathbb{X}$ of the system. Then the system is nonexpansive iff for all $x \in \mathbb{X}$ and all $v \in \mathbb{R}^n$ then whenever $n \in N_v$ we must have
    \[
    n^\top \mathcal{J}_f(x) v \leq 0.
    \]
\end{thm}

\begin{proof}
    Suppose that the system is nonexpansive. Then for all $t \geq 0$ we have that $\| \phi_t(x) - \phi_t(y) \| \leq \|x - y\|$. Thus we must have that for $n \in N_{x-y}$ that
    \[
    n^\top(\phi_t(x) - \phi_t(y)) \leq n^\top (x - y)
    \]
    or
    \[
    n^\top((\phi_t(x) - x) - (\phi_t(y) - y)) \leq 0.
    \]
    Note the first inequality is due to the observation that if $n$ is the normal vector of a supporting hyperplane $H$ of a convex figure, and for $v \in H$ we have $n^{\top} v = c$, then for all points in the convex figure we must have $n^{\top} v \leq c$. 
    
    Now we have that $\phi_t(x) - x = t f(x) + t \epsilon_1(t)$ where $\epsilon_1(t) \rightarrow 0$ as $t \rightarrow 0$, and similarly $\phi_t(y) - y = t f(y) + t \epsilon_2(t)$ where $\epsilon_2(t) \rightarrow 0$ as $t \rightarrow 0$. Thus we get  
    \begin{align*}
    & n^\top(t f(x) + t \epsilon_1(t) - (t f(y) + t \epsilon_2(t)) \\ 
    &= n^\top(t (f(x) - f(y)) + t (\epsilon_1(t)-\epsilon_2(t))) \leq 0
    \end{align*}
    or
    \[
    n^\top( (f(x) - f(y)) + (\epsilon_1(t)-\epsilon_2(t))) \leq 0.
    \]
    Let $t \rightarrow 0$ we get that $n^\top(f(x) - f(y)) \leq 0$. Now we have by the mean value theorem that
    \begin{align*}
    & n^{\top}(f(x) - f(y)) = \int_0^1 n^{\top}(\mathcal{J}_f(y + (x-y)t) (x-y)) dt \\
    &= n^{\top}(\mathcal{J}_f(y + (x-y)s) (x-y)).
    \end{align*}
    Thus also have that
    \begin{align*}
        n^\top(f(x) - f(y))  &\leq 0 \\
        n^\top(\mathcal{J}_f(y + (x-y)s) (x-y) &\leq 0 
    \end{align*}

    Let $x_r = y + (x-y)r$, so that $x_1 = x$ and $x_r-y = (x-y)r$. We can divide by $r$ to get the inequality
    \[
        n^\top(\mathcal{J}_f(y + (x-y)rs) (x-y) \leq 0
    \]
    Letting $r \rightarrow 0$ we get that $(x-y)rs \rightarrow 0$ and so we have that 
    \[
    n^\top \mathcal{J}_f(y) (x-y) \leq 0.
    \]

    This is the desired condition. 
    Now we will prove the other direction. Again note that 
    \begin{align*}
    &n^\top(f(x) - f(y)) = \int_0^1 n^\top(\mathcal{J}_f(y + (x-y)t) (x-y)) dt \\ 
    &= n^\top(\mathcal{J}_f(y + (x-y)s) (x-y)).
    \end{align*}
    Thus we have that
    \begin{align*}
    &n^\top(f(x) - f(y)) = n^\top(\mathcal{J}_f(y + (x-y)s) (x-y)) \leq 0.
    \end{align*}
    The last inequality is by assumption. Thus $n^\top(\dot{x}-\dot{y}) = n^\top(f(x) - f(y)) \leq 0$. From this it follows that the vector $x(t)-y(t)$ is not moving outside of the ball $B_{\|x-y\|}$ and so the system is nonexpansive. 
\end{proof}

\subsection{Examples}
\subsubsection{Systems nonexpansive with respect to the \texorpdfstring{$l^4$}{Lg} norm}
 Using Theorem \ref{thm:nonexpansive demidovich} we can show that there exists systems with nonexpansive $l^p$ norms for $p \neq 1,2, \infty$. For the $l^4$ norm in $\mathbb{R}^2$ the condition from Theorem \ref{thm:nonexpansive demidovich} is that 
 \[
 [u^3,v^3] \mathcal{J}_f(x) \begin{bmatrix} u \\ v \end{bmatrix} \leq 0.
 \] 
  Note that
\begin{align*} 
&-(u^2 + 2cuv -2c^2 v^2)^2 = -u^4 - 4 c u^3 v + 8c^3 u v^3 - 4c^4 v^4 \\
&= [u^3,v^3]\begin{bmatrix}
    -1 & -4c \\ 
    8c^3 & -4c^4 
\end{bmatrix} \begin{bmatrix}
    u \\
    v
\end{bmatrix} \leq 0.
\end{align*}
This implies that for the matrix 
$$A_c =\begin{bmatrix}
    -1 & -4c \\ 
    8c^3 & -4c^4 
\end{bmatrix}.$$ 
that the linear system $\dot{x} = A_cx$ in nonexpansive with respect to the $l^4$ norm for all real numbers $c$. Note also that for $u = (1 + \sqrt{3})cv$ we have that
\[
[u^3,v^3]\begin{bmatrix}
    -1 & -4c \\ 
    8c^3 & -4c^4 
\end{bmatrix} \begin{bmatrix}
    u \\
    v
\end{bmatrix} = 0.
\]

\subsubsection{A globally convergent Hurwitz everywhere system which is not contractive with respect to any norm}

 We can also show that in fact there is a Hurwitz everywhere system that is globally convergent which is not nonexpansive with respect to any norm. Consider the system
 \begin{align*}
     \dot{x} &= -x \\
     \dot{y} &= -(x^2+1)y
 \end{align*}

Now by Theorem \ref{thm:nonexpansive demidovich} for the system to be nonexpansive with respect to a norm we have that 
\[
n^\top \mathcal{J}_f(x,y) v \leq 0
\]
as in the notation of the theorem. Note that for the system under consideration that 
\[
\mathcal{J}_f(x,y) = \begin{bmatrix} -1 & 0 \\ -2xy & -(x^2 + 1) \end{bmatrix}.
\] 
For $v$ with nonzero $x$ coordinate we have that $\mathcal{J}_f(x,y) v$ contains every vector with a negative $x$ coordinate. This forces $n$ to be the vector $[-1,0]$ or some positive multiple of this vector. There does not exist a bounded symmetric convex shape centered at the origin in $\mathbb{R}^2$ such that any supporting hyperplanes at points with nonzero $x$ coordinate have normal $[-1,0]$ (the only such shape with this property would be two parallel lines).

\section{Conclusions}

We characterized the $\omega$-limit sets of (generally nonlinear) nonexpansive dynamical systems with respect to strictly convex norms as attractors of linear systems. A common theme throughout our paper is that the isometry group of a norm is closely tied to the behavior of dynamical systems nonexpansive with respect to the norm. We also provided a complete description of nonexpansive systems in $\mathbb{R}^2$, and presented a Demidovich type condition which we used to provide some examples of nonexpansive systems.

\bibliographystyle{IEEEtran}
\bibliography{nonexpansive_ieee}

\end{document}